\documentclass{amsart}

\usepackage{amsmath,amssymb,amsthm,mathtools}
\usepackage{microtype}
\usepackage[hidelinks]{hyperref}
\usepackage{enumitem}

\newtheorem{theorem}{Theorem}[section]
\newtheorem{proposition}[theorem]{Proposition}
\newtheorem{lemma}[theorem]{Lemma}

\theoremstyle{remark}
\newtheorem{remark}[theorem]{Remark}

\newcommand{\Hh}{\mathfrak H}
\newcommand{\Hone}{\mathfrak H_1}
\newcommand{\E}{\mathbb E}
\newcommand{\Pp}{\mathbb P}
\newcommand{\R}{\mathbb R}
\newcommand{\ip}[2]{\left\langle #1,#2\right\rangle}
\newcommand{\norm}[1]{\left\lVert #1\right\rVert}
\newcommand{\lawge}{\mathrel{\overset{\mathrm{law}}{\ge}}}

\title[Malliavin Smoothness of the Third-Order Hermite Process]
{Malliavin Smoothness of the Third-Order Hermite Process}

\author{Elina Moldavskaya}
\address{Technion---Israel Institute of Technology, Israel}
\email{elina.mol@technion.ac.il}


\begin{document}

\begin{abstract}
We prove Malliavin nondegeneracy for the third-order Hermite process. The key step is to show that, for every nonzero test direction $h\in C_c^\infty(0,1)$, the directional Malliavin derivative of a third-order Hermite random variable is an infinite-rank Gaussian quadratic form. Using arbitrarily large orthonormal families of such directions, we derive a self-contained Fourier bound for their joint characteristic function and obtain polynomial small-ball estimates of arbitrary order for the Malliavin norm. This yields negative moments of every order. Combining the one-time estimate with determinant factorization and Malliavin strong local nondeterminism, we obtain negative moments of all orders for finite-dimensional Malliavin determinants. Consequently, all finite-dimensional distributions, as well as arbitrary vectors of non-overlapping increments, admit Schwartz densities. We further establish grid-uniform Sobolev bounds for inverse Malliavin determinants of normalized increment vectors and derive stretched-exponential estimates for all partial derivatives of their densities. The decay exponent is $2/3$, reflecting the third Wiener chaos. This settles the next non-Gaussian Hermite order after the Rosenblatt case and provides a quantitative counterpart to the order-two smoothness theory.
\end{abstract}

\maketitle

\section{Introduction}

Hermite processes are canonical non-Gaussian self-similar processes with stationary increments arising in non-central limit theorems for nonlinear functionals of long-range dependent Gaussian sequences. The process of order one is fractional Brownian motion, while order two gives the Rosenblatt process. Higher orders provide the natural non-Gaussian limits associated with larger Hermite ranks; see, among others, \cite{DobrushinMajor1979,Taqqu1979,PipirasTaqqu2017,Tudor2023}.

A basic regularity question for these processes concerns the existence and smoothness of their finite-dimensional densities. For a random vector $F=(F_1,\ldots,F_m)$ with Malliavin differentiable components, absolute continuity follows from almost sure invertibility of its Malliavin matrix by the Bouleau--Hirsch criterion. Smoothness is substantially stronger: one needs inverse moments of the determinant of the Malliavin matrix. If the components belong to $\mathbb D^\infty$ and
\[
 (\det \Gamma_F)^{-1}\in \bigcap_{p\ge1}L^p(\Omega),
\]
then the density is smooth; when the components also have moments of every order, the density belongs to the Schwartz class; see, for example, \cite{Nualart2006}.

Very recently, Loosveldt, Nachit, Nourdin and Tudor proved absolute continuity of finite-dimensional distributions for Hermite processes of arbitrary order \cite{LoosveldtEtAlAC2026}. Their argument combines a factorization of the Malliavin determinant with strong local nondeterminism formulated at the level of Malliavin derivatives. In Remark~7.1 of \cite{LoosveldtEtAlAC2026}, the authors point out that almost sure positivity of the Malliavin determinant does not provide smoothness and that the natural next question is whether their strategy can be strengthened to obtain negative moments of the determinant.

For the Rosenblatt process, corresponding to Hermite order two, Loosveldt, Nachit, Nourdin and Tudor subsequently answered this question affirmatively \cite{LoosveldtEtAlRosenblatt}. Their proof exploits the special structure of the second Wiener chaos: the Malliavin derivative of a Rosenblatt random variable is Gaussian, so its squared norm can be represented through a Karhunen--Lo\`eve expansion and controlled by weighted chi-square variables. Negative moments for Malliavin derivatives of general second-chaos variables are also known; see \cite[Lemma~7.1]{HuLuNualart2014}.

The distributional theory of vectors of Gaussian quadratic forms is developed, for example, in \cite{BogachevKosovNourdinPoly2015}, while general absolute-continuity criteria for finite vectors on Wiener chaos are given in \cite{NourdinNualartPoly2013}. A related but distinct line of work concerns regularization inside the third Wiener chaos near a Gaussian limit. Poly \cite{Poly2019} obtains improved Carbery--Wright-type nondegeneracy estimates along central convergence by encoding a Malliavin gradient through a Gaussian random matrix. Those estimates provide a fixed improvement of the small-ball exponent, but do not yield negative moments of every order for the fixed non-Gaussian Hermite law considered here.

The purpose of the present paper is to settle the next Hermite order. Our main result is that the third-order Hermite process is Malliavin nondegenerate in every finite dimension. In particular, all finite-dimensional distributions have Schwartz densities.

The mechanism is specific to order three and differs from the Rosenblatt argument. If $F$ is a third-order Hermite random variable and $h\in C_c^\infty(0,1)$ is deterministic, then
\[
 D_hF\in \mathcal H_2,
\]
so $D_hF$ is a Gaussian quadratic form. We prove that every nonzero such direction gives an infinite-rank form. For any finite orthonormal family $h_1,\ldots,h_m$, compactness of the unit sphere and continuity of singular values then give a uniform lower bound on a fixed singular value of the operators associated with
\[
 D_{\sum_i u_i h_i}F,\qquad u\in\mathbb S^{m-1}.
\]
The explicit characteristic function of a second-chaos variable therefore yields an integrable bound for the joint characteristic function of $(D_{h_1}F,\ldots,D_{h_m}F)$. Fourier inversion gives a bounded joint density and hence
\[
 \Pp\big\{\norm{DF}_{L^2([0,1])}\le \varepsilon\big\}
 \le C_m\varepsilon^m
 \qquad\text{for every }m\ge1.
\]
This self-contained argument is quantitatively stronger than the mere existence of a joint density. It also agrees with the general infinite-rank quadratic-form criterion of Bogachev \cite[Section~5]{Bogachev2016}. The small-ball estimate immediately yields all negative moments of the Malliavin norm at time one. The strong local nondeterminism theorem of \cite{LoosveldtEtAlAC2026} then transfers these inverse-moment estimates to the conditional factors in the Gram determinant of an arbitrary finite-dimensional Malliavin matrix.

The argument highlights why Hermite order three is a distinguished case. For order two, directional Malliavin derivatives are Gaussian. For order three, they are quadratic Gaussian forms, and their second-chaos spectral structure can be exploited directly. Starting from order four, directional derivatives live in Wiener chaoses of order at least three, and the present quadratic-form argument no longer applies directly.

Beyond qualitative smoothness, we obtain a quantitative estimate for increment densities. If $0=t_0<t_1<\cdots<t_m$, $\Delta_j=t_j-t_{j-1}$, and $p_\pi$ denotes the density of the successive increment vector, then for every multi-index $\mathbf n=(n_1,\ldots,n_m)$,
\[
 |\partial^{\mathbf n}p_\pi(x)|
 \le C\prod_{j=1}^m\Delta_j^{-H(1+n_j)}
 \exp\left\{-c\sum_{j=1}^m\left(\frac{x_j}{\Delta_j^H}\right)^{2/3}\right\}
\]
whenever $x_j\ge2\Delta_j^H$. The power $2/3$ is the fixed-chaos tail exponent for the third Wiener chaos. Thus the qualitative nondegeneracy result also yields a grid-uniform analogue of the density-derivative estimates recently obtained for the Rosenblatt process, with the exponential tail replaced by the natural third-chaos stretched exponential.

The paper is organized as follows. Section~\ref{sec:prelim} recalls the Hermite-process representation and the Malliavin tools needed later. In Section~\ref{sec:directional} we prove the infinite-rank property for every nonzero direction in $C_c^\infty(0,1)$. Section~\ref{sec:negative} establishes all negative moments of the Malliavin derivative at time one. Section~\ref{sec:fdd} combines this estimate with the determinant factorization and Malliavin strong local nondeterminism from \cite{LoosveldtEtAlAC2026} to obtain Schwartz densities of all finite-dimensional distributions. Section~\ref{sec:quantitative} gives grid-uniform Sobolev estimates for inverse Malliavin determinants and stretched-exponential bounds for derivatives of increment densities.

\section{Preliminaries and main results}\label{sec:prelim}

Let $B=(B(h),h\in\Hh)$ be an isonormal Gaussian process over
\[
 \Hh=L^2(\R).
\]
We denote by $I_q$ the multiple Wiener--It\^o integral of order $q$ with respect to $B$, and by $D$ the Malliavin derivative. For $F=I_q(f)$ with symmetric $f\in \Hh^{\odot q}$,
\[
 D_rF=qI_{q-1}(f(\cdot,r)).
\]
All standard Malliavin notation follows \cite{Nualart2006}. We write
\[
 \mathbb D^\infty:=\bigcap_{k\ge1}\bigcap_{p\ge1}\mathbb D^{k,p}.
\]
For $k\ge1$ and $p\ge1$, the Sobolev norm on $\mathbb D^{k,p}$ is
\[
 \|F\|_{k,p}^p:=\E[|F|^p]+\sum_{j=1}^k\E\big[\|D^jF\|_{\Hh^{\otimes j}}^p\big],
\]
and $\|\cdot\|_{0,p}$ denotes the norm of $L^p(\Omega)$.
For a random vector $X=(X_1,\ldots,X_n)$ with components in $\mathbb D^{1,2}$, its Malliavin matrix is
\[
 \Gamma_X:=\big(\ip{DX_i}{DX_j}_{\Hh}\big)_{1\le i,j\le n}.
\]
We call $X$ Malliavin nondegenerate if $\det\Gamma_X>0$ almost surely and
\[
 (\det\Gamma_X)^{-1}\in\bigcap_{p\ge1}L^p(\Omega).
\]
The Schwartz space $\mathcal S(\R^n)$ consists of all $C^\infty$ functions whose derivatives decay faster than every inverse polynomial.

Let $Z^{H,3}=(Z_t^{H,3})_{t\in\R}$ be the Hermite process of order three and self-similarity index $H\in(1/2,1)$. Following the normalization of \cite{LoosveldtEtAlAC2026}, write
\[
 H_0=1+\frac{H-1}{3}\in\left(\frac56,1\right),
\]
and
\[
 Z_t^{H,3}=I_3(L_t^{H,3}),
\]
where
\begin{equation}\label{eq:HermiteKernel}
 L_t^{H,3}(\xi_1,\xi_2,\xi_3)
 =c(H,3)\int_0^t\prod_{j=1}^3
 (s-\xi_j)_+^{H_0-3/2}\,ds.
\end{equation}
The positive constant $c(H,3)$ is chosen so that $\E[(Z_1^{H,3})^2]=1$.

For the one-time analysis it is convenient to introduce
\begin{equation}\label{eq:alpha-beta}
 \alpha:=2-2H_0=\frac{2(1-H)}{3}\in\left(0,\frac13\right),
 \qquad
 \beta:=\frac{1+\alpha}{2}=\frac32-H_0\in\left(\frac12,\frac23\right).
\end{equation}
Then
\[
 (s-x)_+^{H_0-3/2}=(s-x)_+^{-\beta}.
\]
We write
\[
 F:=Z_1^{H,3}=I_3(f_\alpha),
\]
with
\begin{equation}\label{eq:falpha}
 f_\alpha(x_1,x_2,x_3)
 =c_H\int_0^1\prod_{j=1}^3(s-x_j)_+^{-\beta}\,ds,
\end{equation}
where $c_H=c(H,3)>0$.

The condition $H>1/2$ enters here exactly through square integrability. Indeed, using \eqref{eq:betaidentity} below,
$$
 \norm{f_\alpha}_{L^2(\R^3)}^2
 =c_H^2 b_\alpha^3
   \int_0^1\int_0^1 |s-t|^{-3\alpha}\,ds\,dt
 =c_H^2 b_\alpha^3
   \frac{2}{(1-3\alpha)(2-3\alpha)}<\infty,
$$
because $3\alpha<1$.

We shall use the restricted Hilbert space
\[
 \Hone:=L^2([0,1])
\]
and the restricted Malliavin norm
\[
 \norm{DF}_{\Hone}^2:=\int_0^1|D_rF|^2\,dr.
\]
This is the same restricted norm that appears on the right-hand side of the Malliavin strong local nondeterminism estimate in \cite[Theorem~6.1]{LoosveldtEtAlAC2026}.

Our one-time result is the following.

\begin{theorem}[Negative moments at order three]\label{thm:negative-main}
For every $p>0$,
\begin{equation}\label{eq:negative-main}
 \E\big[\norm{DF}_{\Hone}^{-p}\big]<\infty.
\end{equation}
More precisely, for every integer $m\ge1$ there exists $C_m<\infty$ such that
\begin{equation}\label{eq:small-ball-main}
 \Pp\{\norm{DF}_{\Hone}\le\varepsilon\}
 \le C_m\varepsilon^m,
 \qquad \varepsilon>0.
\end{equation}
\end{theorem}

The finite-dimensional consequence is stated next. If $\mathbf Z=(Z^{H,3}_{t_1},\ldots,Z^{H,3}_{t_n})$, denote by $\Gamma_{\mathbf Z}$ its Malliavin matrix.

\begin{theorem}[Malliavin nondegeneracy]\label{thm:fdd-main}
Let $n\ge1$ and let $t_1,\ldots,t_n$ be distinct positive times. Then, for every $p>0$,
\begin{equation}\label{eq:det-negative}
 \E\big[(\det\Gamma_{\mathbf Z})^{-p}\big]<\infty.
\end{equation}
Consequently, the law of $\mathbf Z$ admits a density in the Schwartz space $\mathcal S(\R^n)$. The same conclusion holds for every finite vector of pairwise non-overlapping increments.
\end{theorem}

For the quantitative statement, let $0=t_0<t_1<\cdots<t_m$ and write
\[
 \Delta_j:=t_j-t_{j-1},
 \qquad
 X_j:=Z_{t_j}^{H,3}-Z_{t_{j-1}}^{H,3},
 \qquad 1\le j\le m.
\]
Let $p_\pi$ denote the density of $\mathbf X_\pi=(X_1,\ldots,X_m)$. For $\mathbf n=(n_1,\ldots,n_m)\in\mathbb N_0^m$, write
\[
 \partial^{\mathbf n}:=\prod_{j=1}^m\left(\frac{\partial}{\partial x_j}\right)^{n_j},
 \qquad |\mathbf n|:=\sum_{j=1}^m n_j.
\]

\begin{theorem}[Quantitative density bounds]\label{thm:density-bounds}
Let $m\ge1$ and $\mathbf n=(n_1,\ldots,n_m)\in\mathbb N_0^m$. There exist constants $C,c>0$, depending only on $H$, $m$ and $\mathbf n$, such that for every grid $0=t_0<t_1<\cdots<t_m$ and every $x=(x_1,\ldots,x_m)$ satisfying
\[
 x_j\ge2\Delta_j^H,\qquad 1\le j\le m,
\]
one has
\begin{equation}\label{eq:density-derivative-main}
 |\partial^{\mathbf n}p_\pi(x)|
 \le
 C\prod_{j=1}^m\Delta_j^{-H(1+n_j)}
 \exp\left\{-c\sum_{j=1}^m
 \left(\frac{x_j}{\Delta_j^H}\right)^{2/3}\right\}.
\end{equation}
\end{theorem}

\begin{remark}[Symmetry]\label{rem:symmetry}
The map $B\mapsto-B$ preserves the law of the isonormal process and sends $I_3(g)$ to $-I_3(g)$. Hence
\[
 (Z_t^{H,3})_{t\in\R}\overset{\mathrm{law}}=(-Z_t^{H,3})_{t\in\R},
 \qquad\text{so that}\qquad
 p_\pi(-x)=p_\pi(x),\quad x\in\R^m.
\]
Differentiating this identity $\mathbf n$ times gives
\[
 (\partial^{\mathbf n}p_\pi)(-x)
 =(-1)^{|\mathbf n|}(\partial^{\mathbf n}p_\pi)(x),
\]
so that $|(\partial^{\mathbf n}p_\pi)(-x)|=|(\partial^{\mathbf n}p_\pi)(x)|$ and \eqref{eq:density-derivative-main} holds verbatim on the reflected region $x_j\le-2\Delta_j^H$, $1\le j\le m$. Mixed orthants are not covered: the integration by parts of Proposition~\ref{prop:uniform-ibp} produces the indicator of an upper orthant, whose probability is not small when some coordinates are large and negative. The symmetry is specific to odd Hermite orders; the second-chaos Rosenblatt law is not symmetric, so \cite[Theorem~2]{LoosveldtEtAlRosenblatt} has no such extension.
\end{remark}

\section{Directional derivatives and infinite-rank quadratic forms}\label{sec:directional}

We first isolate the deterministic operator structure behind a directional derivative of $F$.

For $u$ initially smooth on $(0,1)$ define
\begin{equation}\label{eq:Ualpha}
 (U_\alpha u)(x)
 :=\int_0^1u(s)(s-x)_+^{-\beta}\,ds,
 \qquad x\in\R.
\end{equation}
For $s\ne t$, the beta-integral identity yields
\begin{equation}\label{eq:betaidentity}
 \int_\R(s-x)_+^{-\beta}(t-x)_+^{-\beta}\,dx
 =b_\alpha|s-t|^{-\alpha},
\end{equation}
where
\[
 b_\alpha=B\left(\frac{1-\alpha}{2},\alpha\right)>0.
\]
Hence, for $u\in L^2(0,1)$,
\begin{equation}\label{eq:U-energy}
 \norm{U_\alpha u}_{L^2(\R)}^2
 =b_\alpha\int_0^1\int_0^1
 u(s)u(t)|s-t|^{-\alpha}\,ds\,dt.
\end{equation}
Since $\alpha<1$, Schur's test shows that $U_\alpha$ extends to a bounded operator from $L^2(0,1)$ to $L^2(\R)$. If
\[
 (R_\alpha u)(s)=\int_0^1|s-t|^{-\alpha}u(t)\,dt,
\]
then
\begin{equation}\label{eq:UstarU}
 U_\alpha^*U_\alpha=b_\alpha R_\alpha.
\end{equation}

\begin{lemma}[Injectivity of $U_\alpha$]\label{lem:Uinjective}
The operator $U_\alpha$ is injective. Consequently,
\[
 \overline{\operatorname{Ran}(U_\alpha^*)}=L^2(0,1).
\]
\end{lemma}

\begin{proof}
Extend $u\in L^2(0,1)$ by zero outside $[0,1]$. The Riesz energy admits the Fourier representation
\[
 \int_\R\int_\R
 u(s)u(t)|s-t|^{-\alpha}\,ds\,dt
 =c_\alpha\int_\R|\widehat u(\xi)|^2|\xi|^{\alpha-1}\,d\xi,
\]
with $c_\alpha>0$. Thus the quadratic form of $R_\alpha$ is strictly positive on nonzero $u$. By \eqref{eq:UstarU}, $U_\alpha u=0$ implies $u=0$. The density statement follows from
\[
 \overline{\operatorname{Ran}(U_\alpha^*)}
 =(\ker U_\alpha)^\perp.
\]
\end{proof}

Let $h\in C_c^\infty(0,1)$ be real-valued and define
\begin{equation}\label{eq:ah}
 a_h(s):=\int_0^s h(r)(s-r)^{-\beta}\,dr,
 \qquad 0<s<1.
\end{equation}
Since $\beta<1$, $a_h\in L^\infty(0,1)$.

\begin{lemma}[Injectivity of the fractional integral]\label{lem:ah-nonzero}
If $h\in C_c^\infty(0,1)$ is nonzero, then $a_h$ is nonzero on a set of positive Lebesgue measure.
\end{lemma}

\begin{proof}
Up to the positive normalizing factor $\Gamma(1-\beta)$,
\[
 a_h=I_{0+}^{1-\beta}h,
\]
where $I_{0+}^{\gamma}$ denotes the Riemann--Liouville fractional integral of order $\gamma$. If $a_h=0$ almost everywhere, applying $I_{0+}^{\beta}$ and using the semigroup property gives
\[
 I_{0+}^{1}h=0.
\]
Hence $\int_0^s h(r)\,dr=0$ for almost every $s$, and therefore $h=0$, a contradiction.
\end{proof}

The directional Malliavin derivative along $h$ is
\[
 D_hF:=\ip{DF}{h}_{L^2([0,1])}.
\]
For $r\in(0,1)$, the $L^2(\R^2)$-norm of the slice $f_\alpha(\cdot,\cdot,r)$ satisfies
\[
 \norm{f_\alpha(\cdot,\cdot,r)}_2^2
 =c_H^2b_\alpha^2
 \int_r^1\int_r^1
 (s-r)^{-\beta}(t-r)^{-\beta}|s-t|^{-2\alpha}\,ds\,dt.
\]
After the scaling $s-r=(1-r)u$, $t-r=(1-r)v$, the right-hand side is bounded by
\[
 C_{\alpha,H}(1-r)^{1-3\alpha},
\]
because $\beta<1$, $2\alpha<1$, and $\beta+\alpha<1$. Hence
\[
 \int_0^1 |h(r)|\,
 \norm{f_\alpha(\cdot,\cdot,r)}_2\,dr<\infty.
\]
Thus the integral below is a Bochner integral in $L^2(\R^2)$, and the $L^2$-continuity of $I_2$ yields
\begin{equation}\label{eq:DhF}
 D_hF=3I_2(k_h),
\end{equation}
where
\begin{equation}\label{eq:kh}
 k_h(y,z)
 =c_H\int_0^1
 a_h(s)(s-y)_+^{-\beta}(s-z)_+^{-\beta}\,ds.
\end{equation}
Moreover,
\[
 \norm{k_h}_2^2
 =c_H^2b_\alpha^2
 \int_0^1\int_0^1
 a_h(s)a_h(t)|s-t|^{-2\alpha}\,ds\,dt<\infty,
\]
since $a_h\in L^\infty(0,1)$ and $2\alpha<1$. Let $A_h$ denote the Hilbert--Schmidt operator on $L^2(\R)$ associated with $k_h$.

\begin{lemma}[Operator factorization]\label{lem:factorization}
For every real $h\in C_c^\infty(0,1)$,
\begin{equation}\label{eq:factorization}
 A_h=c_HU_\alpha M_{a_h}U_\alpha^*,
\end{equation}
where $M_{a_h}$ is multiplication by $a_h$ on $L^2(0,1)$.
\end{lemma}

\begin{proof}
Take first $v,w\in C_c^\infty(\R)$. For every compact set $K\subset\R$,
\[
 \sup_{s\in(0,1)}\int_K |s-x|^{-\beta}\,dx<\infty
\]
because $\beta<1$. Thus the expressions defining $U_\alpha^*v$ and $U_\alpha^*w$ are pointwise finite, and the integrand below is absolutely integrable. By \eqref{eq:kh} and ordinary Fubini,
$$
 \ip{A_hv}{w}_{L^2(\R)}
 =c_H\int_0^1a_h(s)
 \left(\int_\R(s-y)_+^{-\beta}v(y)\,dy\right)
 \times
 \left(\int_\R(s-z)_+^{-\beta}w(z)\,dz\right)ds
 $$
 $$
 =c_H\int_0^1a_h(s)\,
 (U_\alpha^*v)(s)(U_\alpha^*w)(s)\,ds.
$$
Because $a_h\in L^\infty(0,1)$ and $U_\alpha$ is bounded, the right-hand side is the bounded bilinear form associated with
\[
 c_HU_\alpha M_{a_h}U_\alpha^*.
\]
Since $C_c^\infty(\R)$ is dense in $L^2(\R)$, the identity extends to all $v,w\in L^2(\R)$.
\end{proof}

\begin{proposition}[Infinite rank for smooth compactly supported directions]\label{prop:infinite-rank}
If $0\ne h\in C_c^\infty(0,1)$, then the operator $A_h$ has infinite rank. Equivalently, the quadratic Gaussian form $D_hF$ has infinite rank.
\end{proposition}

\begin{proof}
By Lemma~\ref{lem:ah-nonzero}, $a_h\ne0$ on a set of positive measure. Hence there exists $\delta>0$ such that
\[
 E_\delta:=\{s\in(0,1):|a_h(s)|\ge\delta\}
\]
has positive measure. Since Lebesgue measure is non-atomic, $E_\delta$ contains infinitely many pairwise disjoint measurable subsets of positive measure. Their indicator functions are mapped by $M_{a_h}$ to linearly independent functions. Thus $M_{a_h}$ has infinite rank.

Suppose that $A_h$ had finite rank. Since $U_\alpha$ is injective and
\[
 A_h=c_HU_\alpha M_{a_h}U_\alpha^*,
\]
the space $M_{a_h}\operatorname{Ran}(U_\alpha^*)$ would be finite-dimensional and hence closed. By Lemma~\ref{lem:Uinjective},
\[
 \overline{\operatorname{Ran}(U_\alpha^*)}=L^2(0,1).
\]
Because $M_{a_h}$ is continuous,
\[
 M_{a_h}(L^2(0,1))
 =M_{a_h}\big(\overline{\operatorname{Ran}(U_\alpha^*)}\big)
 \subseteq
 \overline{M_{a_h}\operatorname{Ran}(U_\alpha^*)}
 =M_{a_h}\operatorname{Ran}(U_\alpha^*).
\]
Hence $M_{a_h}$ would have finite rank, a contradiction.
\end{proof}

\begin{remark}\label{rem:quadratic}
Let $k$ be a symmetric Hilbert--Schmidt kernel and let $(\lambda_j)$ be the eigenvalues of its associated Hilbert--Schmidt operator. Then
\[
 I_2(k)\overset{\mathrm{law}}=
 \sum_{j\ge1}\lambda_j(\xi_j^2-1),
\]
where $(\xi_j)$ are independent standard Gaussian variables; see, for example, \cite[Proposition~2.7.13]{NourdinPeccati2012}. Hence the rank of the quadratic Gaussian form is exactly the rank of the associated Hilbert--Schmidt operator.
\end{remark}

\section{Small balls and all negative moments}\label{sec:negative}

We now derive the boundedness of the joint density of finitely many directional derivatives directly from their characteristic function. This avoids any uniformity issue hidden in a qualitative density theorem and yields an explicit bound.

\begin{proposition}[Bounded joint density of directional derivatives]\label{prop:bounded-density}
Fix an integer $m\ge1$ and choose real orthonormal functions
\[
 h_1,\ldots,h_m\in C_c^\infty(0,1).
\]
Set
\[
 Q_i:=D_{h_i}F,\qquad 1\le i\le m.
\]
Then $(Q_1,\ldots,Q_m)$ has a bounded continuous density $\rho_m$ on $\R^m$.

More precisely, let $N>2m$, and for $u=(u_1,\ldots,u_m)\in\mathbb S^{m-1}$ write
\[
 h_u:=\sum_{i=1}^m u_i h_i.
\]
If $s_N(A_{h_u})$ denotes the $N$-th singular value of $A_{h_u}$, then
\[
 c_{m,N}:=\min_{u\in\mathbb S^{m-1}}s_N(A_{h_u})>0,
\]
and
\begin{equation}\label{eq:rho-bound}
 \norm{\rho_m}_\infty
 \le
 \frac{\pi^{m/2}}{(2\pi)^m(6c_{m,N})^m}
 \frac{\Gamma(N/4-m/2)}{\Gamma(N/4)}.
\end{equation}
\end{proposition}

\begin{proof}
The map
\[
 u\longmapsto A_{h_u}=\sum_{i=1}^m u_iA_{h_i}
\]
is continuous from $\mathbb S^{m-1}$ to the Hilbert--Schmidt operators, hence also in operator norm. Singular values are Lipschitz with respect to the operator norm. By Proposition~\ref{prop:infinite-rank}, $A_{h_u}$ has infinite rank for every $u\in\mathbb S^{m-1}$, so $s_N(A_{h_u})>0$. Compactness of $\mathbb S^{m-1}$ therefore gives $c_{m,N}>0$.

Let
\[
 \varphi_m(t)
 :=\E\exp\left(i\sum_{j=1}^m t_jQ_j\right),
 \qquad t\in\R^m.
\]
For $t\ne0$, write $t=ru$ with $r=|t|$ and $u\in\mathbb S^{m-1}$. Let $(\lambda_j(u))_{j\ge1}$ be the eigenvalues of the self-adjoint Hilbert--Schmidt operator $A_{h_u}$, repeated with multiplicity and ordered so that $|\lambda_1(u)|\ge|\lambda_2(u)|\ge\cdots$. Then Remark~\ref{rem:quadratic} gives
\[
 \sum_{j=1}^m t_jQ_j
 =3r I_2(k_{h_u})
 \overset{\mathrm{law}}=
 3r\sum_{j\ge1}\lambda_j(u)(\xi_j^2-1).
\]
Consequently,
\begin{equation}\label{eq:cf-product}
 |\varphi_m(t)|
 =
 \prod_{j\ge1}
 \bigl(1+36r^2\lambda_j(u)^2\bigr)^{-1/4}.
\end{equation}
Keeping only the first $N$ factors and using
$|\lambda_j(u)|\ge s_N(A_{h_u})\ge c_{m,N}$ for $1\le j\le N$, we obtain
\begin{equation}\label{eq:cf-decay}
 |\varphi_m(t)|
 \le
 \bigl(1+36c_{m,N}^2|t|^2\bigr)^{-N/4}.
\end{equation}
Since $N>2m$, the right-hand side is integrable on $\R^m$. Fourier inversion therefore gives a bounded continuous density and
\[
 \norm{\rho_m}_\infty
 \le
 (2\pi)^{-m}
 \int_{\R^m}
 \bigl(1+36c_{m,N}^2|t|^2\bigr)^{-N/4}\,dt.
\]
The standard beta-integral evaluation of the last integral gives
\eqref{eq:rho-bound}.
\end{proof}

\begin{remark}\label{rem:Bogachev-alternative}
The preceding proposition is also consistent with the general result
\cite[Section~5]{Bogachev2016}: measurable Gaussian quadratic forms whose
nontrivial linear combinations have infinite rank possess a joint density in
$\mathcal S(\R^m)$.  We do not use that theorem in the proof of the small-ball
estimate, because Proposition~\ref{prop:bounded-density} makes the required
uniformity over directions explicit.
\end{remark}

\begin{proposition}[Arbitrarily high small-ball powers]\label{prop:small-ball}
For every integer $m\ge1$ there exists $C_m<\infty$ such that
\[
 \Pp\{\norm{DF}_{\Hone}\le\varepsilon\}
 \le C_m\varepsilon^m,
 \qquad \varepsilon>0.
\]
\end{proposition}

\begin{proof}
Choose real orthonormal functions
\[
 h_1,\ldots,h_m\in C_c^\infty(0,1)
\]
and set $Q_i=D_{h_i}F$. By Proposition~\ref{prop:bounded-density}, the vector
$(Q_1,\ldots,Q_m)$ has a bounded density $\rho_m$. Bessel's inequality gives
\[
 \sum_{i=1}^mQ_i^2\le\norm{DF}_{\Hone}^2.
\]
Hence
\begin{align*}
 \Pp\{\norm{DF}_{\Hone}\le\varepsilon\}
 &\le
 \Pp\left\{\sum_{i=1}^mQ_i^2\le\varepsilon^2\right\}\\
 &\le
 \norm{\rho_m}_\infty\,\omega_m\varepsilon^m,
\end{align*}
where $\omega_m$ is the volume of the unit ball in $\R^m$.
\end{proof}

\begin{proof}[Proof of Theorem~\ref{thm:negative-main}]
The small-ball estimate is Proposition~\ref{prop:small-ball}. Fix $p>0$ and choose an integer $m>p$. With $X=\norm{DF}_{\Hone}$, the layer-cake formula gives
\[
 \E[X^{-p}]
 =p\int_0^\infty t^{p-1}\Pp\{X<t^{-1}\}\,dt.
\]
The integral over $(0,1]$ is finite. On $[1,\infty)$, Proposition~\ref{prop:small-ball} gives an integrand bounded by $C_mt^{p-1-m}$, which is integrable because $m>p$.
\end{proof}

\section{Finite-dimensional nondegeneracy}\label{sec:fdd}

We now use the two structural results from \cite{LoosveldtEtAlAC2026}. First, let
\[
 \mathbf X=(X_1,\ldots,X_n)\in(\mathbb D^{1,2})^n
\]
and let $\Gamma_{\mathbf X}$ be its Malliavin matrix. Then \cite[Lemma~4.1]{LoosveldtEtAlAC2026} gives the Gram-factorization
\begin{equation}\label{eq:GramFactor}
 \det\Gamma_{\mathbf X}
 =\norm{DX_1}_{\Hh}^2
 \prod_{j=2}^n
 \norm{DX_j-\operatorname{proj}_{E_{j-1}}(DX_j)}_{\Hh}^2,
\end{equation}
where $E_{j-1}=\operatorname{span}\{DX_1,\ldots,DX_{j-1}\}$.

Second, for a Hermite process of arbitrary order, \cite[Theorem~6.1]{LoosveldtEtAlAC2026} establishes the Malliavin strong local nondeterminism estimate. For $0=t_0<t_1<\cdots<t_j$,
\begin{equation}\label{eq:SLND}
 \norm{DZ_{t_j}^{H,3}-\operatorname{proj}_{E_{j-1}}(DZ_{t_j}^{H,3})}_{\Hh}^2
 \lawge
 (t_j-t_{j-1})^{2H}\norm{DZ_1^{H,3}}_{\Hone}^2.
\end{equation}
Here $X\lawge Y$ means first-order stochastic domination:
\[
 \E[\varphi(X)]\ge\E[\varphi(Y)]
\]
for every bounded non-decreasing $\varphi:\R_+\to\R$.

We first record the elementary inverse-moment consequence of this order.

\begin{lemma}[Stochastic domination and inverse moments]\label{lem:stochastic-inverse}
If $X,Y>0$ and $X\lawge cY$ for some $c>0$, then, for every $r>0$,
\[
 \E[X^{-r}]\le c^{-r}\E[Y^{-r}]
\]
whenever the right-hand side is finite.
\end{lemma}

\begin{proof}
For $M>0$, let
\[
 g_M(x)=\min\{x^{-r},M\}.
\]
Then $-g_M$ is bounded and non-decreasing. Stochastic domination yields
\[
 \E[g_M(X)]\le\E[g_M(cY)].
\]
Let $M\to\infty$ and apply monotone convergence.
\end{proof}

\begin{proof}[Proof of Theorem~\ref{thm:fdd-main}]
Reorder the times so that
\[
 0<t_1<\cdots<t_n.
\]
Set $\mathbf Z=(Z_{t_1}^{H,3},\ldots,Z_{t_n}^{H,3})$. By \eqref{eq:GramFactor},
\[
 \det\Gamma_{\mathbf Z}
 =Y_1\cdots Y_n,
\]
where
\[
 Y_1=\norm{DZ_{t_1}^{H,3}}_{\Hh}^2
\]
and, for $j\ge2$,
\[
 Y_j=\norm{DZ_{t_j}^{H,3}-\operatorname{proj}_{E_{j-1}}(DZ_{t_j}^{H,3})}_{\Hh}^2.
\]

The Malliavin self-similarity established in \cite[Proposition~5.1]{LoosveldtEtAlAC2026}, together with $\norm{DZ_1^{H,3}}_{\Hh}\ge\norm{DZ_1^{H,3}}_{\Hone}$, shows that $Y_1$ has negative moments of every order by Theorem~\ref{thm:negative-main}. For $j\ge2$, \eqref{eq:SLND}, Lemma~\ref{lem:stochastic-inverse}, and Theorem~\ref{thm:negative-main} imply
\[
 \E[Y_j^{-r}]<\infty
 \qquad\text{for every }r>0.
\]
Therefore, for fixed $p>0$, H\"older's inequality gives
\[
 \E[(\det\Gamma_{\mathbf Z})^{-p}]
 =\E\left[\prod_{j=1}^nY_j^{-p}\right]
 \le\prod_{j=1}^n\E[Y_j^{-pn}]^{1/n}<\infty.
\]
This proves \eqref{eq:det-negative}.

Each $Z_{t_j}^{H,3}$ belongs to a fixed Wiener chaos and therefore to $\mathbb D^\infty$ and to $L^r$ for every $r\ge1$. The standard Malliavin nondegeneracy criterion then yields a density in the Schwartz space; see, for instance, \cite[Proposition~2.1.5]{Nualart2006}, or the formulation recalled in \cite[Theorem~3]{LoosveldtEtAlRosenblatt}.

For successive increments
\[
 (Z_{t_1}^{H,3},Z_{t_2}^{H,3}-Z_{t_1}^{H,3},\ldots,Z_{t_n}^{H,3}-Z_{t_{n-1}}^{H,3}),
\]
the conclusion follows from the result for values by an invertible triangular linear transformation.

More generally, let
\[
 (Z_{b_1}^{H,3}-Z_{a_1}^{H,3},\ldots,
  Z_{b_r}^{H,3}-Z_{a_r}^{H,3})
\]
be a vector of pairwise non-overlapping increments. Collect all distinct positive endpoints among the $a_i,b_i$ and order them as
\[
 0<s_1<\cdots<s_M.
\]
The vector of increments is a surjective linear image of
$(Z_{s_1}^{H,3},\ldots,Z_{s_M}^{H,3})$. After a linear change of coordinates, the density of such an image is obtained by integrating a Schwartz density over the kernel variables. Partial integration preserves the Schwartz class. Hence every finite vector of pairwise non-overlapping increments has a density in $\mathcal S$.
\end{proof}

\section{Quantitative estimates for increment densities}\label{sec:quantitative}

We now strengthen Theorem~\ref{thm:fdd-main} by keeping track of the dependence on the time grid. The normalization of successive increments makes the argument particularly transparent.

Fix $m\ge1$ and a grid $0=t_0<t_1<\cdots<t_m$. Put $\Delta_j=t_j-t_{j-1}$ and
\[
 X_j:=Z_{t_j}^{H,3}-Z_{t_{j-1}}^{H,3},
 \qquad
 \widehat X_j:=\Delta_j^{-H}X_j,
 \qquad
 \widehat{\mathbf X}_\pi=(\widehat X_1,\ldots,\widehat X_m).
\]
Each $\widehat X_j$ belongs to the third Wiener chaos, has variance one, and has the same one-dimensional law as $Z_1^{H,3}$.

\begin{proposition}[Uniform inverse-determinant Sobolev bounds]\label{prop:uniform-det}
For all integers $k\ge0$, $m\ge1$ and every $p\ge1$, there exists $C=C(H,m,k,p)<\infty$ such that
\begin{equation}\label{eq:normalized-det-Sobolev}
 \sup_{0=t_0<\cdots<t_m}
 \big\|(\det\Gamma_{\widehat{\mathbf X}_\pi})^{-1}\big\|_{k,p}
 \le C.
\end{equation}
Consequently,
\begin{equation}\label{eq:unnormalized-det-Sobolev}
 \big\|(\det\Gamma_{\mathbf X_\pi})^{-1}\big\|_{k,p}
 \le
 C\prod_{j=1}^m\Delta_j^{-2H}.
\end{equation}
\end{proposition}

\begin{proof}
Let $\widetilde E_{j-1}$ be the span of $DX_1,\ldots,DX_{j-1}$. Since the transformation from values $(Z_{t_1},\ldots,Z_{t_{j-1}})$ to successive increments $(X_1,\ldots,X_{j-1})$ is triangular and invertible,
\[
 \widetilde E_{j-1}
 =\operatorname{span}\{DZ_{t_1}^{H,3},\ldots,DZ_{t_{j-1}}^{H,3}\}.
\]
Moreover, $DZ_{t_{j-1}}^{H,3}\in\widetilde E_{j-1}$, and hence
\[
 DX_j-\operatorname{proj}_{\widetilde E_{j-1}}DX_j
 =DZ_{t_j}^{H,3}-\operatorname{proj}_{\widetilde E_{j-1}}DZ_{t_j}^{H,3}.
\]
The Gram factorization and the Malliavin strong local nondeterminism estimate therefore imply, after normalization by $\Delta_j^H$, that each conditional factor in $\det\Gamma_{\widehat{\mathbf X}_\pi}$ stochastically dominates a copy of
$\norm{DZ_1^{H,3}}_{\Hone}^2$. The first factor is handled by Malliavin self-similarity in the same way. Theorem~\ref{thm:negative-main}, Lemma~\ref{lem:stochastic-inverse}, and H\"older's inequality yield
\begin{equation}\label{eq:uniform-det-Lp}
 \sup_{\pi}
 \E[(\det\Gamma_{\widehat{\mathbf X}_\pi})^{-r}]<\infty
 \qquad\text{for every }r>0.
\end{equation}

For positive Sobolev norms, write $\widehat X_j=I_3(g_j)$ with $g_j\in\Hh^{\odot3}$. Then $3!\,\|g_j\|_{\Hh^{\otimes3}}^2=\E[\widehat X_j^2]=1$, and for $\ell\le3$ the derivative $D^\ell\widehat X_j=\frac{3!}{(3-\ell)!}I_{3-\ell}(g_j)$ is an $\Hh^{\otimes\ell}$-valued element of the Wiener chaos of order $3-\ell$, while $D^\ell\widehat X_j=0$ for $\ell\ge4$. Consequently, for $0\le\ell\le3$,
\[
 \E\big[\|D^\ell\widehat X_j\|_{\Hh^{\otimes\ell}}^2\big]
 =\left(\frac{3!}{(3-\ell)!}\right)^2(3-\ell)!\,
 \|g_j\|_{\Hh^{\otimes3}}^2
 =\frac{3!}{(3-\ell)!},
\]
which depends neither on $j$ nor on the grid. Let $p\ge2$. For almost every $r=(r_1,\ldots,r_\ell)\in\R^\ell$ the variable $D_r^\ell\widehat X_j$ is a scalar element of the Wiener chaos of order $3-\ell$, so the hypercontractivity estimate $\|I_q(f)\|_{L^p(\Omega)}\le(p-1)^{q/2}\|I_q(f)\|_{L^2(\Omega)}$ (see \cite[Theorem~2.7.2]{NourdinPeccati2012} or \cite[Section~1.4.3]{Nualart2006}) applies to it. Combining this with Minkowski's integral inequality in $L^{p/2}(\Omega)$ gives
\begin{align*}
 \Big(\E\big[\|D^\ell\widehat X_j\|_{\Hh^{\otimes\ell}}^p\big]\Big)^{2/p}
 &=\Big\|\int_{\R^\ell}|D_r^\ell\widehat X_j|^2\,dr\Big\|_{L^{p/2}(\Omega)}
 \le\int_{\R^\ell}\big\|D_r^\ell\widehat X_j\big\|_{L^p(\Omega)}^2\,dr\\
 &\le(p-1)^{3-\ell}
 \int_{\R^\ell}\big\|D_r^\ell\widehat X_j\big\|_{L^2(\Omega)}^2\,dr
 =(p-1)^{3-\ell}\,\frac{3!}{(3-\ell)!}.
\end{align*}
Only the scalar hypercontractivity estimate has been used. For $p<2$ one uses monotonicity of $L^p$ norms. Hence, for all integers $r\ge1$ and all $p\ge1$,
\[
 \sup_{\pi}\max_{1\le j\le m}
 \|\widehat X_j\|_{r,p}\le C(r,p)<\infty,
\]
with $C(r,p)$ independent of the grid. Hence every entry of the Malliavin matrix of $\widehat{\mathbf X}_\pi$, its determinant, and all their Malliavin derivatives have moments of every order uniformly in the grid. Iterating the chain and product rules from $D(\det\Gamma)^{-1}=-(\det\Gamma)^{-2}D(\det\Gamma)$, one obtains by induction on $\ell$ that $D^\ell(\det\Gamma)^{-1}$ is a finite linear combination of terms
\[
 (\det\Gamma)^{-1-i}\,
 D^{a_1}(\det\Gamma)\otimes\cdots\otimes D^{a_i}(\det\Gamma),
 \qquad 1\le i\le\ell,\quad a_1+\cdots+a_i=\ell;
\]
compare \cite[Lemma~2.1.6]{Nualart2006}. Combining these terms with \eqref{eq:uniform-det-Lp} by H\"older's inequality proves \eqref{eq:normalized-det-Sobolev}.

Finally,
\[
 \Gamma_{\mathbf X_\pi}
 =\operatorname{diag}(\Delta_1^H,\ldots,\Delta_m^H)
 \Gamma_{\widehat{\mathbf X}_\pi}
 \operatorname{diag}(\Delta_1^H,\ldots,\Delta_m^H),
\]
so
\[
 (\det\Gamma_{\mathbf X_\pi})^{-1}
 =\left(\prod_{j=1}^m\Delta_j^{-2H}\right)
 (\det\Gamma_{\widehat{\mathbf X}_\pi})^{-1}.
\]
The deterministic prefactor commutes with Malliavin differentiation, and \eqref{eq:unnormalized-det-Sobolev} follows.
\end{proof}

\begin{proposition}[Uniform integration-by-parts weights]\label{prop:uniform-ibp}
Let $\mathbf n=(n_1,\ldots,n_m)\in\mathbb N_0^m$. For every grid $\pi$ there exists $H_{\pi,\mathbf n}\in\mathbb D^\infty$ such that the density $\widehat p_\pi$ of $\widehat{\mathbf X}_\pi$ satisfies
\begin{equation}\label{eq:density-IBP-representation}
 \partial^{\mathbf n}\widehat p_\pi(y)
 =(-1)^{|\mathbf n|}
 \E\left[\mathbf 1_{\{\widehat X_1\ge y_1,\ldots,\widehat X_m\ge y_m\}}
 H_{\pi,\mathbf n}\right],
\end{equation}
and
\begin{equation}\label{eq:H-uniform}
 \sup_{\pi}\|H_{\pi,\mathbf n}\|_2<\infty.
\end{equation}
\end{proposition}

\begin{proof}
The standard iterative Malliavin integration-by-parts formula for a nondegenerate vector gives \eqref{eq:density-IBP-representation}; see, for example, \cite[Propositions~2.1.4--2.1.5]{Nualart2006}. The corresponding weights are obtained recursively from the inverse Malliavin matrix, Malliavin derivatives of the vector, and the divergence operator. Meyer's inequalities therefore bound their $L^2$ norms by finitely many Sobolev norms of $(\det\Gamma_{\widehat{\mathbf X}_\pi})^{-1}$ and of the components $\widehat X_j$; an explicit bound of this type, with $|\mathbf n|+m$ in place of the number of integrations by parts, is \cite[Lemma~3.6]{BoufoussiNachit2023}, as used in the proof of \cite[Proposition~2]{LoosveldtEtAlRosenblatt}. Proposition~\ref{prop:uniform-det} and the uniform fixed-chaos Sobolev bounds used in its proof yield \eqref{eq:H-uniform}.
\end{proof}

\begin{proof}[Proof of Theorem~\ref{thm:density-bounds}]
By the fixed-chaos tail estimate, there exists $c_3>0$ such that for every third-chaos random variable $Y$ and every $u\ge2$,
\[
 \Pp\left\{\frac{|Y|}{(\E Y^2)^{1/2}}\ge u\right\}
 \le \exp(-c_3u^{2/3});
\]
see, for instance, \cite[Lemma~1]{LoosveldtEtAlRosenblatt}. Since each $\widehat X_j$ has variance one, H\"older's inequality gives, for $y_j\ge2$,
$$
 \Pp\{\widehat X_j\ge y_j,\ 1\le j\le m\}
 \le
 \prod_{j=1}^m\Pp\{|\widehat X_j|\ge y_j\}^{1/m}\\
 \le
 \exp\left\{-\frac{c_3}{m}
 \sum_{j=1}^m y_j^{2/3}\right\}.
$$
Combining this estimate with \eqref{eq:density-IBP-representation}, \eqref{eq:H-uniform}, and Cauchy--Schwarz yields
\begin{equation}\label{eq:normalized-density-bound}
 |\partial^{\mathbf n}\widehat p_\pi(y)|
 \le C
 \exp\left\{-c\sum_{j=1}^m y_j^{2/3}\right\},
 \qquad y_j\ge2,
\end{equation}
with constants independent of the grid.

Finally, the exact scaling relation
\[
 p_\pi(x)
 =\left(\prod_{j=1}^m\Delta_j^{-H}\right)
 \widehat p_\pi\left(
 \frac{x_1}{\Delta_1^H},\ldots,
 \frac{x_m}{\Delta_m^H}\right)
\]
implies
\[
 \partial^{\mathbf n}p_\pi(x)
 =\left(\prod_{j=1}^m\Delta_j^{-H(1+n_j)}\right)
 (\partial^{\mathbf n}\widehat p_\pi)\left(
 \frac{x_1}{\Delta_1^H},\ldots,
 \frac{x_m}{\Delta_m^H}\right).
\]
Applying \eqref{eq:normalized-density-bound} proves \eqref{eq:density-derivative-main}.
\end{proof}

\begin{remark}
For the Rosenblatt process the corresponding fixed-chaos tail exponent is $2/q=1$, and the analogue of \eqref{eq:density-derivative-main} is genuinely exponential; see \cite[Theorem~2]{LoosveldtEtAlRosenblatt}. At Hermite order three the natural exponent is $2/3$. Proposition~\ref{prop:uniform-det} shows that the inverse-determinant Sobolev scaling in the mesh sizes is nevertheless identical: the only change in the final density estimate is the tail profile of the underlying chaos.
\end{remark}

\section{Why the argument is specific to Hermite order three}\label{sec:order-three}

The proof above exploits a feature that is exact at the third Hermite order. If $Z^{H,q}$ is a Hermite process of order $q$ and $h$ is deterministic, then
\[
 D_hZ_1^{H,q}\in\mathcal H_{q-1}.
\]
For $q=2$, this is the first Wiener chaos, so the directional derivative is Gaussian. This is the setting exploited in the Rosenblatt smoothness argument of \cite{LoosveldtEtAlRosenblatt}. For $q=3$, the directional derivative lies in the second Wiener chaos and can be treated as a Gaussian quadratic form; the self-contained Fourier argument of Section~\ref{sec:negative} then yields arbitrarily high small-ball powers. For $q\ge4$, the directional derivative has Gaussian polynomial degree at least three, and the quadratic-form theorem no longer applies directly.

It is worth comparing this with regularization results for general random variables in the third Wiener chaos. Poly \cite{Poly2019} studies regularization along convergence to a nondegenerate Gaussian law. In the third chaos he improves the Carbery--Wright small-ball exponent by a factor three by introducing a Malliavin gradient whose law is encoded through the spectrum of a Gaussian matrix. This still yields a fixed polynomial nondegeneracy exponent and does not provide negative moments of arbitrary order for a fixed non-Gaussian target. The present setting is different: $F=Z_1^{H,3}$ is a fixed non-central Hermite limit, and the specific integral structure of its kernel implies that every nonzero directional derivative from $C_c^\infty(0,1)$ is an infinite-rank second-chaos form. Compactness over finite-dimensional spheres then upgrades this pointwise infinite-rank property to the arbitrarily high small-ball exponents in Proposition~\ref{prop:small-ball}.

For comparison, negative moments of Malliavin derivatives of general second-chaos variables are available by \cite[Lemma~7.1]{HuLuNualart2014}; this fact is also emphasized in \cite[Remark~1]{LoosveldtEtAlRosenblatt}. Thus the difficulty addressed here is genuinely the passage from a third-chaos target to a sufficiently nondegenerate family of its second-chaos directional derivatives.

For $q\ge4$, one might try to differentiate repeatedly until reaching the second chaos. This does not directly solve the small-ball problem for the first Malliavin derivative: Bessel's inequality controls $\norm{DZ_1^{H,q}}$ from below by first directional derivatives, whereas higher directional derivatives introduce additional random coefficients and do not provide a deterministic lower bound for $\norm{DZ_1^{H,q}}$. Moreover, general fractional smoothness results for polynomial images, such as those surveyed in \cite[Section~5]{Bogachev2016}, yield only finite Besov regularity and do not provide the bounded joint densities in arbitrarily large dimension that are used in Proposition~\ref{prop:small-ball}. This identifies the precise obstruction to extending the present argument beyond order three.

\begin{remark}
The argument proves considerably more than almost sure positivity of the Malliavin derivative at time one. Estimate \eqref{eq:small-ball-main} holds with an arbitrarily large polynomial exponent. In this sense the order-three Hermite derivative is strongly nondegenerate, despite being non-Gaussian and belonging to the second chaos after directional differentiation.
\end{remark}

\end{document}